\documentclass{amsart}
\usepackage{amsmath,amscd,amsthm,amsfonts,amssymb}

\theoremstyle{plain}
\newtheorem{theorem}                 {Theorem}      [section]
\newtheorem{proposition}  [theorem]  {Proposition}

\newtheorem{lemma}        [theorem]  {Lemma}

\theoremstyle{definition}
\newtheorem{example}      [theorem]  {Example}

\newtheorem{definition}   [theorem]  {Definition}

\numberwithin{equation}{section}

\def \theo-intro#1#2 {\vskip .25cm\noindent{\bf Theorem #1\ }{\it #2}}

\def \rn{\mathbb R}

\def \F{\mathcal F}
\def \H{\mathcal H}

\def \V{\mathcal V}

\def \ip #1#2{\langle #1,#2 \rangle}

\def \lb#1#2{[#1,#2]}

\def \g{\mathfrak{g}}

\def \nab#1#2{\hbox{$\nabla$\kern -.3em\lower 1.0 ex
    \hbox{$#1$}\kern -.1 em {$#2$}}}

\begin{document}
\baselineskip 22pt \larger

\allowdisplaybreaks

\title
{Holomorphic harmonic morphisms from cosymplectic
almost Hermitian manifolds}

\author{Sigmundur Gudmundsson}



\dedicatory{version 2.017 - 3 September 2014}

\address
{Department of Mathematics, Faculty of Science, Lund University,
Box 118, S-221 00 Lund, Sweden}

\email
{Sigmundur.Gudmundsson@math.lu.se}

\begin{abstract}
We study 4-dimensional orientable Riemannian manifolds equipp\-ed
with a minimal and conformal foliation $\F$ of codimension 2.
We prove that the two adapted almost Hermitian structures
$J_1$ and $J_2$ are both cosymplectic if and only if $\F$
is Riemannian and its horizontal distribution $\H$ is
integrable.
\end{abstract}

\maketitle

\section{Introduction}

The notion of a minimal submanifold of a given ambient space is of
great importance in differential geometry. Harmonic morphisms
$\phi:(M,g)\to(N,h)$ between Riemannian manifolds are useful
tools for the construction of such objects. They are solutions to
over-determined non-linear systems of partial differential
equations determined by the geometric data of the manifolds
involved. For this reason harmonic morphisms are difficult to find
and have no general existence theory, not even locally.

For the existence of harmonic morphisms $\phi:(M,g)\to (N,h)$ it
is an advantage that the target manifold $N$ is a surface i.e. of
dimension $2$. In this case the problem is invariant under
conformal changes of the metric on $N^2$. Therefore, at least for
local studies, the codomain can be taken to be the complex plane
with its standard flat metric.

In this paper we are interested in 4-dimensional orientable Riemannian
manifolds $(M^4,g)$ equipp\-ed with a minimal and conformal foliation
$\F$ of codimension 2.
These are important since they produce local complex-valued
harmonic morphisms on $M$, see Section \ref{morphisms-foliations}.
Our following main result, gives a new relationship between the
geometry of the foliation $\F$ and the cosymplecticity
of both its adapted almost Hermitian structures.

\begin{theorem}\label{theo-main-1}
Let $(M^4,g)$ be a 4-dimensional orientable Riemannian manifold equipp\-ed
with a minimal and conformal foliation $\F$ of codimension 2.  Then the
corresponding adapted almost Hermitian structures $J_1$ and $J_2$ are both
cosymplectic if and only if $\F$ is Riemannian and its horizontal
distribution $\H$ is integrable.
\end{theorem}

For the general theory of harmonic morphisms between Riemannian
manifolds we refer to the excellent book \cite{Bai-Woo-book}
and the regularly updated on-line bibliography \cite{Gud-bib}.

\section{Harmonic morphisms and minimal conformal
foliations}\label{morphisms-foliations}

Let $M$ and $N$ be two manifolds of dimensions $m$ and $n$,
respectively. A Riemannian metric $g$ on $M$ gives rise to the
notion of a {\it Laplacian} on $(M,g)$ and real-valued {\it
harmonic functions} $f:(M,g)\to\rn$. This can be generalized to
the concept of {\it harmonic maps} $\phi:(M,g)\to (N,h)$ between
Riemannian manifolds, which are solutions to a semi-linear system
of partial differential equations, see \cite{Bai-Woo-book}.

\begin{definition}
  A map $\phi:(M,g)\to (N,h)$ between Riemannian manifolds is
  called a {\it harmonic morphism} if, for any harmonic function
  $f:U\to\rn$ defined on an open subset $U$ of $N$ with $\phi^{-1}(U)$
non-empty,
  $f\circ\phi:\phi^{-1}(U)\to\rn$ is a harmonic function.
\end{definition}

The following characterization of harmonic morphisms between
Riemannian manifolds is due to Fuglede and T. Ishihara.  For the
definition of horizontal (weak) conformality we refer to
\cite{Bai-Woo-book}.

\begin{theorem}\cite{Fug-1,T-Ish}
  A map $\phi:(M,g)\to (N,h)$ between Riemannian manifolds is a
  harmonic morphism if and only if it is a horizontally (weakly)
  conformal harmonic map.
\end{theorem}

Let $(M,g)$ be a Riemannian manifold, $\V$ be an involutive
distribution on $M$ and denote by $\H$ its orthogonal
complement distribution on $M$.
As customary, we also use $\V$ and $\H$ to denote the
orthogonal projections onto the corresponding subbundles of $TM$
and denote by $\F$ the foliation tangent to
$\V$. The second fundamental form for $\V$ is given by
$$B^\V(U,V)=\frac 12\H(\nabla_UV+\nabla_VU)\qquad(U,V\in\V),$$
while the second fundamental form for $\H$ satisfies
$$B^\H(X,Y)=\frac{1}{2}\V(\nabla_XY+\nabla_YX)\qquad(X,Y\in\H).$$
The foliation $\F$ tangent to $\V$ is said to be {\it conformal} if there is a
vector field $V\in \V$ such that $$B^\H=g\otimes V,$$ and
$\F$ is said to be {\it Riemannian} if $V=0$.
Furthermore, $\F$ is said to be {\it minimal} if $\text{trace}\ B^\V=0$ and
{\it totally geodesic} if $B^\V=0$. This is equivalent to the
leaves of $\F$ being minimal and totally geodesic submanifolds
of $M$, respectively.

It is easy to see that the fibres of a horizontally conformal
map (resp.\ Riemannian submersion) give rise to a conformal foliation
(resp.\ Riemannian foliation). Conversely, the leaves of any
conformal foliation (resp.\ Riemannian foliation) are
locally the fibres of a horizontally conformal map
(resp.\ Riemannian submersion), see \cite{Bai-Woo-book}.

The next result of Baird and Eells gives the theory of
harmonic morphisms, with values in a surface,
a strong geometric flavour.

\begin{theorem}\cite{Bai-Eel}\label{theo:B-E}
Let $\phi:(M^m,g)\to (N^2,h)$ be a horizontally conformal
submersion from a Riemannian manifold to a surface. Then $\phi$ is
harmonic if and only if $\phi$ has minimal fibres.
\end{theorem}

\section{cosymplectic almost Hermitian structures}\label{section-3}

An almost Hermitian manifold $(M,g,J)$ is said to be {\it cosymplectic}
if its almost complex structure $J$ is {\it divergence-free} i.e.
$$\delta J_k=\text{div}J=\sum_{k=1}^m(\nab {X_k}{J})(X_k)=0,$$
where $\{X_1,\dots ,X_m\}$ is any local orthonormal frame for the
tangent bundle $TM$ of $M$.  As an application of a well-known result
from \cite{L} of A. Lichnerowicz, we have the following useful result.

\begin{proposition}\label{prop-surf-1}\cite{Gud-Woo-2}
Let $\phi:(M,g,J)\to N$ be a holomorphic map from an almost Hermitian
manifold to a Riemann surface.  Then $\phi$ is a harmonic morphism if
and only if $d\phi(J\delta J)=0$.
\end{proposition}

In the light of the above discussion, the result of Proposition 
\ref{prop-surf-1} has an equivalent formulation in terms of foliations.

\begin{proposition}\label{prop-surf-2}
Let $(M,g,J)$ be an almost Hermitian manifold and $\F$ be a holomorphic
minimal conformal foliation on $M$ of codimension 2.  Then $\F$ produces 
harmonic morphisms on $M$ if and only if the divergence $\delta J$ of
the almost Hermitian structure $J$ is vertical i.e. $\H\delta J=0$.
\end{proposition}

We will now assume that $(M^4,g)$ is a 4-dimensional orientable Riemannian
manifold equipped with a minimal and conformal foliation $\F$ of
codimension 2.  Then there exist, up to sign, exactly two almost
Hermitian structure $J_1$ and $J_2$ on $M$ which are adapted to the
orthogonal decomposition $TM=\V\oplus\H$
of the tangent bundle of $M$.  They are determined by
$$J_1X=Y,\ J_1Y=-X,\ J_1Z=W,\ J_1W=-Z,$$
$$J_2X=Y,\ J_2Y=-X,\ J_2Z=-W,\ J_2W=Z,$$
where $\{X,Y,Z,W\}$ is any local orthonormal frame for the tangent
bundle $TM$ of $M$ such that $X,Y\in\H$ and $Z,W\in\V$, respectively.

We are now ready to prove our main result stated in Theorem \ref{theo-main-1}.

\begin{proof}
Let us assume that the almost complex structures $J_1$ and
$J_2$ are both cosymplectic i.e. for $k=1,2$ we have
\begin{eqnarray*}
0&=&\delta J_k\\
&=&(\nab X{J_k})(X)+(\nab Y{J_k})(Y)+(\nab Z{J_k})(Z)+(\nab W{J_k})(W)\\
&=&\lb XY+(-1)^k\lb WZ-J_k(\nab XX+\nab YY+\nab ZZ+\nab WW).
\end{eqnarray*}
It now follows from Proposition \ref{prop-surf-2} that
\begin{eqnarray*}
0&=&\delta J_1+\delta J_2\\
&=&\V\delta J_1+\V\delta J_2\\
&=&2\V [X,Y]-\V(J_1+J_2)(\nab XX+\nab YY+\nab ZZ+\nab WW)\\
&=&2\V [X,Y].
\end{eqnarray*}
This shows that the horizontal distribution $\H$ is integrable.
Then employing the fact that $J_1$ is cosymplectic, we see that
\begin{eqnarray*}
& &J_1\V (\nab XX+\nab YY)\\
&=&\V J_1(\nab XX+\nab YY)\\
&=&-\V\lb WZ-\V J_1(\nab ZZ+\nab WW)\\
&=&-\ip{\nab WZ}WW+\ip{\nab ZW}ZZ-J_1\V (\nab ZZ+\nab WW)\\
&=&\ip Z{\nab WW}W-\ip W{\nab ZZ}Z-J_1(\ip{\nab ZZ}WW+\ip{\nab WW}ZZ)\\
&=&0.
\end{eqnarray*}
Further it follows from $\V[X,Y]=0$ and $\V(\nab XX+\nab YY)=0$ that
$\V\delta J_2=0$ is equivalent to
$$\V \lb WZ-\V J_2(\nab ZZ+\nab WW)=0.$$
The fact that $\F$ is conformal implies that for each $X\in\H$
\begin{eqnarray*}
2B^\H(X,X)&=&B^\H(X,X)+B^\H(Y,Y)\\
&=&\V (\nab XX+\nab YY)\\
&=&0.
\end{eqnarray*}
Since the second fundamental form $B^\H$ of the horizontal
distribution $\H$ is symmetric the polar identity tells us
that $B^\H\equiv 0$, so $\F$ is Riemannian.

It is easily seen from the above calculations that the other
part of the statement is also valid.
\end{proof}

\section{Examples}

Let $G$ be a 4-dimensional Lie group equipped with a left-invariant
Riemannian metric.  Let $\g$ be the Lie algebra of $G$ and
$\{X,Y,Z,W\}$ be an orthonormal basis for $\g$.  Let $Z,W\in\g$
generate a 2-dimensional left-invariant and integrable distribution $\V$
on $G$ which is conformal and with minimal leaves.  We denote
by $\H$ the horizontal distribution, orthogonal to $\V$, generated by
$X,Y\in\g$.  Then it is easily seen that the basis $\{ X,Y,Z,W\}$ can
be chosen so that the Lie bracket relations for $\g$ are of the form
\begin{eqnarray*}\label{system}
\lb WZ&=&\lambda W,\\
\lb ZX&=&\alpha X +\beta Y+z_1 Z+w_1 W,\\
\lb ZY&=&-\beta X+\alpha Y+z_2 Z+w_2 W,\\
\lb WX&=&     a X     +b Y+z_3 Z-z_1W,\\
\lb WY&=&    -b X     +a Y+z_4 Z-z_2W,\\
\lb YX&=&     r X         +\theta_1 Z+\theta_2 W
\end{eqnarray*}
with real coefficients.  It should be noted that these constants must
be chosen in such a way that the Lie brackets for $\g$ satisfy the
Jacobi identity.  The solutions to that problem were recently classified
in \cite{Gud-Sve-6}. The following easy result describes the geometry
of the situation.

\begin{proposition}\label{prop-geometry}
Let $G$ be a 4-dimensional Lie group and $\{X,Y,Z,W\}$ be an
orthonormal basis for its Lie algebra as  above.  Then
\begin{enumerate}
\item[(i)] $\F$ is {\it totally geodesic} if and only if
  $z_1=z_2=z_3+w_1=z_4+w_2=0$,
\item[(ii)] $\F$ is {\it Riemannian} if and only if $\alpha=a=0$, and
\item[(iii)] $\H$ is {\it integrable} if and only if $\theta_1=\theta_2=0$.
\end{enumerate}
\end{proposition}

The following lemma turns out to be useful later on.

\begin{lemma}
For the above situation we have the following:
\begin{itemize}
\item[i.] The almost Hermitian structure $J_1$ is
cosymplectic if and only if $$\theta_1-2a=0\ \ \text{and}\ \  \theta_2+2\alpha=0.$$
\item[ii.] The almost Hermitian structure $J_2$ is
cosymplectic if and only if $$\theta_1+2a=0\ \ \text{and}\ \  \theta_2-2\alpha=0.$$
\end{itemize}
\end{lemma}

\begin{proof}
A standard calculation involving the Koszul fomula
$$2\ip{\nab XY}Z=\ip{\lb ZX}Y+\ip{\lb ZY}X+\ip Z{\lb XY}$$
shows that for the Levi-Civita connection of $(G,g)$ we have
$$\nab XX=rY+\alpha Z+aW,\ \ \nab YY=\alpha Z+aW,$$
$$\nab ZZ=-z_1X-z_2Y,\ \ \nab WW=z_1X+z_2Y-\lambda Z.$$
Then the divergence of the almost complex structure $J_1$ is given by
\begin{eqnarray*}
\delta J_1&=&\lb XY-\lb WZ-J_1(\nab XX+\nab YY+\nab ZZ+\nab WW)\\
&=&-(\theta_1-2a)Z-(\theta_2+2\alpha)W.
\end{eqnarray*}
This proves i. and ii. is obtained in exactly the same way.
\end{proof}

In the case when $\lambda=0$, $r\neq 0$ and $(a\beta-\alpha b)\neq 0$
the solutions are given by the following 5-dimensional family
$\g_{5}(\alpha,a,\beta,b,r)$, see Case (C) of \cite{Gud-Sve-6}.

\begin{example}[$\g_{5}(\alpha,a,\beta,b,r)$]\label{exam-C1}
\begin{eqnarray*}
\lb ZX&=&\alpha X +\beta Y
+\frac{r(\beta b-\alpha a)}{2(a\beta-\alpha b)} Z+\frac{r(\alpha^2-\beta^2)}{2(a\beta-\alpha b)} W,\\
\lb ZY&=&-\beta X+\alpha Y
+\frac{r(\alpha b+\beta a)}{2(a\beta-\alpha b)} Z-\frac{r\alpha\beta}{(a\beta-\alpha b)} W,\\
\lb WX&=&     a X     +b Y
+\frac{r(b^2-a^2)}{2(a\beta-\alpha b)} Z+\frac{r(\alpha a-\beta b)}{2(a\beta-\alpha b)}W,\\
\lb WY&=&    -b X     +a Y
+\frac{rab }{(a\beta-\alpha b)} Z-\frac{r(\alpha b+\beta a)}{2(a\beta-\alpha b)}W,\\
\lb YX&=&     r X
-\frac{ar^2}{2(a\beta-\alpha b)} Z+\frac{\alpha r^2}{2(a\beta-\alpha b)} W.
\end{eqnarray*}
Since $r\neq 0$ and $\alpha^2+a^2\neq 0$ each of the induced foliations
$\F$ is neither Riemannian nor does it have an integrable horizontal distribution.
This tells us that at most one of the almost Hermitian structures is
cosymplectic.
\end{example}

In the case when $\lambda=0$, $r=0$ and $\alpha b-a\beta =0$ we have
several interesting families of solutions, see Case (F) of \cite{Gud-Sve-6}.
Two of those are presented below.

\begin{example}
If we assume $\alpha=a=0$ and $\beta\neq 0\neq b$ then we obtain the family $\g_{18}(\beta,b,z_3,z_4,\theta_1,\theta_2)$ of the following form
\begin{eqnarray*}
\lb ZX&=& \beta Y+\frac {\beta z_3}b Z-\frac {\beta^2z_3}{b^2} W,\\
\lb ZY&=&-\beta X+\frac {\beta z_4}b Z-\frac {\beta^2z_4}{b^2} W,\\
\lb WX&=& b Y+z_3 Z-\frac {\beta z_3}b W,\\
\lb WY&=&-b X+z_4 Z-\frac {\beta z_4}b W,\\
\lb YX&=&\theta_1 Z+\theta_2 W.
\end{eqnarray*}
The corresponding foliations $\F$ are all Riemannian so the almost Hermitian
structures $J_1$ and $J_2$ are both cosymplectic if and only if the horizontal
distribution $\H$ is integrable i.e. $\theta_1^2=\theta_2^2=0$.
\end{example}

\begin{example}
If we assume $\alpha\neq 0\neq a$ and $\beta\neq 0\neq b$ then we get the family
$\g_{20}(\alpha,a,\beta,w_1,w_2)$ of the following form
\begin{eqnarray*}
\lb ZX&=&\alpha X+\beta Y-\frac {aw_1}\alpha Z+w_1 W,\\
\lb ZY&=&-\beta X+\alpha Y-\frac {aw_2}\alpha Z+w_2 W,\\
\lb WX&=&a X+\frac{\beta a}\alpha Y-\frac {a^2w_1}{\alpha^2} Z+\frac a\alpha w_1 W,\\
\lb WY&=&-\frac{\beta a}\alpha X+a Y-\frac {a^2w_2}{\alpha^2} Z+\frac a\alpha w_2 W.
\end{eqnarray*}
The almost Hermitian structure $J_1$ is cosymplectic if and only if
$2\alpha^2+aw_2=0$. The same applies to $J_2$ if and only if $2\alpha^2-aw_2=0$.
It is clear that in none of the cases is the foliation $\F$ Riemannian.  The
horizontal distribution $\H$ is integrable in all the cases.
\end{example}

\section{Integrable almost Hermitian structures}

We conclude this paper with Theorem \ref{theo-main-2} giving 
another relationship between the geometry of the horizontal 
conformal foliation $\F$ and conditions on the almost Hermitian 
structure $J_1$ and $J_2$.  The result follows from Lemma 3.6 
(iii) and Proposition 3.9 of \cite{Woo-1992}, but here we give 
a more direct proof.

\begin{theorem}\label{theo-main-2}
Let $(M^4,g)$ be a 4-dimensional orientable Riemannian manifold equipp\-ed
with a minimal and conformal foliation $\F$ of codimension 2.  Then the
corresponding adapted almost Hermitian structures $J_1$ and $J_2$ are both
integrable if and only if $\F$ is totally geodesic.
\end{theorem}

\begin{proof}
Here we use the same notation as in Section \ref{section-3}.  
For $k=1,2$ the skew-symmetric Nijenhuis tensor $N_k$ of the 
almost Hermitian structure $J_k$ is given by
$$N_k(E,F)=[E,F]+J_k[J_kE,F]+J_k[E,J_kF]-[J_kE,J_kF].$$
It is easily seen that this satisfies
$$N_k(X,Y)=N_k(Z,W)=0$$
and that the horizontal conformality of $\F$ is equivalent to 
$$\H N_k(X,Z)=\H N_k(X,W)=\H N_k(Y,Z)=\H N_k(Y,W)=0.$$
This means that $J_1$ and $J_2$ are both integrable if and only if for $k=1,2$
$$\V N_k(X,Z)=\V N_k(X,W)=\V N_k(Y,Z)=\V N_k(Y,W)=0.$$
We define the 1-forms $\alpha,\beta:C^\infty (\H)\to\rn$ by
$$\alpha(E)=2\ip{B^\V(Z,Z)-B^\V(W,W)}E,$$
$$\beta(E)=2\ip{B^\V(Z,W)+B^\V(W,Z)}E.$$
Since we are assuming that that the fibres are minimal i.e.
$$B^\V(Z,Z)+B^\V(W,W)=0$$
we see that the foliation $\F$ is totally geodesic if and only if
$\alpha$ and $\beta$ vanish.
Now a standard calculation shows that
$$\ip{(N_1+N_2)(X,Z)}Z=\alpha(X)=-\ip{(N_1+N_2)(X,W)}W,$$
$$-\ip{(N_1-N_2)(X,Z)}Z=\beta(Y)=\ip{(N_1-N_2)(X,W)}W,$$

$$\ip{(N_1+N_2)(X,Z)}W=\beta(X)=\ip{(N_1+N_2)(X,W)}Z,$$
$$\ip{(N_1-N_2)(X,Z)}W=\alpha(Y)=\ip{(N_1-N_2)(X,W)}Z,$$

$$\ip{(N_1+N_2)(Y,Z)}Z=\alpha(Y)=-\ip{(N_1+N_2)(Y,W)}W,$$
$$\ip{(N_1-N_2)(Y,Z)}Z=\beta(X)=-\ip{(N_1-N_2)(Y,W)}W,$$

$$\ip{(N_1+N_2)(Y,Z)}W=\beta(Y)=\ip{(N_1+N_2)(Y,W)}Z,$$
$$\ip{(N_1-N_2)(Y,Z)}W=-\alpha(X)=\ip{(N_1-N_2)(Y,W)}Z,$$
The statement of Theorem \ref{theo-main-2} is a direct consequence of 
these equations.

\end{proof}

\end{document}